\documentclass[a4paper, 11pt,reqno]{amsart}

\usepackage[T1]{fontenc}      % Uses 'T1' font encoding.
\usepackage[british]{babel}   % Support for different languages.
\usepackage{csquotes}         % Adds quotation marks according to language.
\usepackage[babel]{microtype} % Makes things look nicer (spacing and so on).

\usepackage{mathtools}        % Extension of 'amsmath' which fixes some bugs.
\usepackage{amssymb}          % Extra symbols and fonts.
\usepackage{amsthm}           % Helps to define "theorem-like" environments.
\usepackage{amstext}          % Defines a '\text' macro for math environments.
\usepackage{mleftright}       % Fixes spacing issue.
\usepackage{gensymb}          % Allows to use degrees.
\usepackage[makeroom]{cancel} % Defines cancellation of terms in formulae.
\usepackage{mathdots}         % Adds some definitions of dots.
\usepackage{mathrsfs}         % Adds rsfs fonts.
\usepackage{dsfont}           % Allows some fancy numbers (\mathds).
\usepackage[shortlabels]{enumitem}
\usepackage{verbatim}
\usepackage{stickstootext}
\usepackage[stickstoo,varbb]{newtxmath} % These two lines are an alternative for stix2 which do essentially the same but fix several issues with some symbols...
\makeatletter
\DeclareFontFamily{U}{ntxmia}{\skewchar \font =127}
 \DeclareFontShape{U}{ntxmia}{m}{it}{
                        <-> \ntxmath@scaled ntxmia
                      }{}    
                      \DeclareFontShape{U}{ntxmia}{b}{it}{
                        <-> \ntxmath@scaled ntxbmia
                      }{}
\makeatother

\usepackage{graphicx}         % Provides commands to include graphics.
\usepackage{psfrag}           % Gives some extra useful things for graphics.
\usepackage{caption}          % Options for formatting float captions.
\usepackage{tikz}             % Language to create graphics programmatically.
\usetikzlibrary{backgrounds}

\usepackage{booktabs}         % Enhances the quality of tables.

\usepackage{enumitem}         % Allows customization for enumerations.

\usepackage[margin=1in]{geometry} % Provides commands to define the page layout.

\usepackage[textsize=footnotesize]{todonotes}
\usepackage[numbers,sort]{natbib}

\makeatletter
\def\NAT@spacechar{~}% NEW
\makeatother
\newcommand{\urlprefix}{}

\usepackage{hyperref}              % For hyperlinks within the document
\usepackage[capitalise]{cleveref}  % For improved referencing options; use [capitalise] if desired
\hypersetup{colorlinks=true,
   citecolor=blue,
   filecolor=blue,
   linkcolor=blue,
   urlcolor=blue
}
\usepackage{url}

\makeatletter
\if@cref@capitalise
\crefname{figure}{figure}{figures}
\crefname{claim}{Claim}{Claims}
\else
\crefname{figure}{Figure}{Figures}
\crefname{claim}{claim}{claims}
\fi
\makeatother
\Crefname{figure}{Figure}{Figures}
\Crefname{claim}{Claim}{Claims}

\usepackage{algorithm}
\usepackage{algpseudocode}

\allowdisplaybreaks

\theoremstyle{plain}
\newtheorem{definition}{Definition}

\newtheorem{theorem}[definition]{Theorem}
\newtheorem{corollary}[definition]{Corollary}
\newtheorem{lemma}[definition]{Lemma}

\newtheorem{conjecture}[definition]{Conjecture}

\numberwithin{equation}{section}

\renewcommand{\binom}[2]{\ensuremath{\mleft(\kern-.1em\genfrac{}{}{0pt}{}{#1}{#2}\kern-.1em\mright)}}    % This makes binomial numbers nicer with stix2 (in displayed equations). Remove if stix2 is not loaded.
\newcommand{\inbinom}[2]{\ensuremath{\bigl(\kern-.1em\genfrac{}{}{0pt}{}{#1}{#2}\kern-.1em\bigr)}} % This is better for inline equations, as it will keep sizes of parentheses consistent and not create extra vertical space.

\DeclareMathOperator\supp{supp}

\newcommand{\PP}{\mathbb{P}}

\newcommand{\EE}{\mathbb{E}}

\newcommand{\lgr}{\left<g\right>}

\newcommand{\tk}{\tilde{k}}
\newcommand{\tm}{\tilde{m}}

\newcommand{\tn}{\tilde{n}}
\newcommand{\ty}{\tilde{y}}

\newcommand{\tY}{\tilde{Y}}

\newcommand{\NN}{\mathbb{N}}

\makeatletter
\def\moverlay{\mathpalette\mov@rlay}
\def\mov@rlay#1#2{\leavevmode\vtop{%
  \baselineskip\z@skip \lineskiplimit-\maxdimen
  \ialign{\hfil$\m@th#1##$\hfil\cr#2\crcr}}}
\newcommand{\charfusion}[3][\mathord]{
    #1{\ifx#1\mathop\vphantom{#2}\fi
        \mathpalette\mov@rlay{#2\cr#3}
      }
    \ifx#1\mathop\expandafter\displaylimits\fi}
\makeatother

\newcommand{\COMMENT}[1]{}
\renewcommand{\COMMENT}[1]{\footnote{\textcolor{blue!70!black}{#1}}} % comment out to hide comments
\newcommand{\COMNEW}[1]{}
\renewcommand{\COMNEW}[1]{\footnote{\textcolor{red!70!black}{#1}}} % comment out to hide comments

\title{Further remarks on fractional vs. expectation thresholds}

\author[T.~Fischer]{Thomas Fischer}
\author[Y.~Person]{Yury Person}

\address{Institut f\"ur Mathematik, Technische Universit\"at Ilmenau, 98684 Ilmenau, Germany} 

\email{thomas.fischer \,|\, yury.person@tu-ilmenau.de}

\thanks{Research is supported by DFG grant 447645533.}

\date{\today}
\begin{document}

\begin{abstract}
A conjecture of Talagrand (2010) states that the so-called expectation and fractional expectation thresholds are always within at most some constant factor from each other. In this note we generalize a method of DeMarco and Kahn in~\cite{TalagrandCliques} and settle a few more special cases.
\end{abstract}

\maketitle

\section{Introduction}
Let $X$ be a finite nonempty set, $k\in\NN$ some positive integer with $k\leq |X|$ and $p\in[0,1]$. We denote by $2^X$ the power set of $X$ and by $\binom{X}{k}$ the set of all $k$-element subsets of $X$. Given a function $g\colon 2^X\to[0,1]$, we will be interested in a subset $\left<g\right>$ of $2^X$ defined as follows:
\begin{equation}\label{eq:family_g}
\left< g\right>:=\left\{S\in 2^X \left|\  \sum_{T\subseteq S} g(T) \geq 1 \right. \right\},
\end{equation}
and we define the weight of $g$ with respect to $p$ as
\[
w(g,p):=\sum_{S\in 2^X} g(S)p^{|S|}.
\]
If $g$ takes only values from $\{0,1\}$ then we can identify $g$ with an indicator function $\mathds{1}_G$ for some $G\subseteq 2^X$, so that we write $\left<G\right>$ instead of  $\left<\mathds{1}_G\right>$ and $\left<G\right>$ thus denotes the  upset of $G$. Moreover, $w(\mathds{1}_G,p)$ will be denoted simply by $w(G,p)=\sum_{S\in G} p^{|S|}$. This latter term $w(G,p)$ has an interpretation as the expected number of sets from $G$ which are contained in a binomial random set $X_p$. The set $X_p$ is  formed by including every $x\in X$ into $X_p$ independently with probability $p$. The quantity $w(g,p)$ may then be viewed as a fractional analogue of the quantity $w(G,p)$. A fascinating conjecture of M.~Talagrand~\cite{TalagrandOriginal} states, in an equivalent version from~\cite{FP23}, the following.
\begin{conjecture}[Conjecture~6 from~\cite{FP23}]\label{conj:Talagrand}
    There exists some fixed $L>1$ such that for all $n\in\NN$, $g: 2^X \rightarrow [0,1]$ with $g(\emptyset)=0$ and $p\in [0,1]$ the following holds. 
 If     $w(g,p) = 1$ then there exists a set $G\subseteq 2^X\setminus\{\emptyset\}$ with $\left<g\right>\subseteq \left<G\right>$ and $w\left(G,\frac{p}{L}\right) \leq 1$.   
\end{conjecture}
The quantities $w(g,p)$, resp.\ $w(G,p)$, are related to the so-called fractional expectation resp.\ expectation thresholds. 
For more background on expectation thresholds and Talagrand's conjecture we refer the reader to~\cite{FP23,TalagrandTwoSets,TalagrandOriginal}. 
 
Talagrand~\cite{TalagrandOriginal} proved Conjecture~\ref{conj:Talagrand} for functions $g$ whose support $\supp(g):=\{S\colon g(S)\neq 0\}$ is contained in $\binom{X}{1}$ and also for functions $g$ so that, for some set $J\subseteq X$, all sets $S$ from $\lgr$ contain at least $(2e)p|J|$ elements from $J$ (see~\cite[Lemma~5.9]{TalagrandOriginal}). 
Talagrand~\cite{TalagrandOriginal} also suggested two further special cases as test cases: when $g$ is constant and supported by the edge sets of the cliques of some order $k$ in the complete graph $K_n$ and when $g$ is supported by a subset of $2$-sets in $X$ (i.e.\ $\supp(g)\subseteq \binom{X}{2}$). The former case was verified by DeMarco and Kahn in~\cite{TalagrandCliques} for graph cliques of any order $k$ and the latter was verified by Frankston, Kahn and Park in~\cite{TalagrandTwoSets}. We generalized in~\cite{FP23} the result from~\cite{TalagrandTwoSets} to `almost' linear hypergraphs of any uniformity, and we also proved that it is enough to study functions $g\colon \binom{X}{k}\to[0,1]$ for all $k\le \ln(|X|)$. 

More recently, Dubroff, Kahn and Park~\cite{DKP24} and Pham~\cite{Pham24} proved that Conjecture~\ref{conj:Talagrand} is true for any function $g\colon 2^X\to[0,1]$ which is supported on sets of constant size. Moreover, from the quantitative result of Pham~\cite{Pham24}, combined with our result~\cite[Theorem~10]{FP23}, it follows that $L$ can be chosen as $O(\log\log|X|)$, which is better than the $O(\log|X|)$ bound implied by the resolution of the Kahn-Kalai conjecture~\cite{KK07} by Park and Pham~\cite{PP24}.

\subsection{An abstract setting for $k$-uniform hypergraphs}
The purpose of this paper is to investigate a few more cases when sets from the support of $g$ are not necessarily of constant cardinality. From our work in~\cite[Theorems~8 and~10]{FP23} it follows that it is sufficient to study the following conjecture whose resolution would imply Conjecture~\ref{conj:Talagrand}. 

\begin{conjecture}\label{conj:Talagrand_rest}
    There exists some fixed $L>1$ such that for all finite sets $X$ and all $k\in\NN$ with $k\leq \log_L\left(\left|X\right|\right)$ and all $g: \binom{X}{k} \rightarrow [0,1]$ with $p\in [0,1]$ the following holds:
    
    If $w(g,p) = 1$ then there exists a set $G\subseteq 2^X\setminus\{\emptyset\}$ with $\left<g\right>\subseteq \left<G\right>$ and $w\left(G,\frac{p}{L}\right) \leq 1$.   
\end{conjecture}

It will be useful to think of the function $g$ as corresponding to a weighted $k$-uniform hypergraph on $X$ where the weighted edges correspond to the sets from the support $\supp(g)$.
The challenge with Conjecture~\ref{conj:Talagrand} (and also with Conjecture~\ref{conj:Talagrand_rest}) is to  construct a $G$ out of $g$. Our main focus in this paper will be a generalization of the approach pioneered by DeMarco and Kahn~\cite{TalagrandCliques} for the case of cliques in graphs. 

We will seek to formulate abstract conditions which are especially satisfied by the hypergraphs whose edges correspond to the edge sets of hypergraph cliques of some order.  
First we provide a general setup which we are going to deal with. 
\begin{definition}[\emph{General assumptions}]\label{def:general_as}
Let $n$, $k\in\NN$ with $\ln n\ge k\ge 2$. Let $X$ be a finite $n$-element set. Let $L>1$, $r\ge L^k$ and $d\in \NN$ with $r\le d$.  Let $g\colon \binom{X}{k}\to[0,1]$ be given and $p\in[0,1]$ such that $w(g,p)=1$ and $g$ be constant and equal $\frac{1}{r}$ on its support $E:=\supp(g)$. 

Now we construct G randomly by choosing $s,t\in\mathbb{N}$ and, for $i\in[t]$, $j\in [s]$, picking independently and uniformly at random sets $e_{i,j}\in E$ (with replacement). Next we define the sets $e_i:=\bigcup_{j\in[s]} e_{i,j}$ for each $i\in [t]$ and finally the (random) set $G:=\left\{e_i \mid i\in[t]\right\}$. Additionally we define for each $j\in[s]$ the random variable $Y_j:=\left|e_{1,j}\cap \bigcup_{j'<j} e_{1,j'}\right|$. 
\end{definition}

The following two technical theorems provide sufficient conditions on $(Y_j)_{j\in [s]}$ to infer 
the existence of $G\subseteq 2^X\setminus\{\emptyset\}$ with $\left<g\right>\subseteq \left<G\right>$ and $w\left(G,\frac{p}{L}\right)\leq 1$.

\begin{theorem}\label{thm:DMK_one}
Assume general assumptions (from Definition~\ref{def:general_as}). Let, additionally, $L\geq 2\cdot e$, $t=p^{-s\cdot k}\cdot n$ and $s\cdot k\geq \ln(n)$ hold. If we have 
\begin{equation}\label{eq:thm_one}
    \mathbb{P}\left(\sum_{j\leq s} Y_j= m\right) \left(\frac{d\cdot e^k}{r}\right)^{\frac{m}{k}} \leq \left(\frac{L}{2e}\right)^{\left(s-1\right)\cdot k-m}
\end{equation}
     for all $m\in[(s-1)\cdot k]$ then there exists $G\subseteq 2^X\setminus\{\emptyset\}$ with $\left<g\right>\subseteq \left<G\right>$ and $w\left(G,\frac{p}{L}\right)\leq 1$.\\
In particular, inequality~\eqref{eq:thm_one} is implied if, for all $y\in[k]$, the following holds
\begin{equation}\label{eq:Y_cond}
\mathbb{P}\left(Y_s \geq y\mid Y_1,\ldots,Y_{s-1}\right) \left(\frac{d\cdot 2^k\cdot e^k}{r}\right)^{\frac{y}{k}} \le \left(\frac{L}{4e}\right)^{k-y}.
\end{equation}
\end{theorem}

The next theorem poses only conditions on the overlap $|e_{1,1}\cap e_{1,2}|$, where we recall that $e_{1,1}$ and $e_{1,2}$ are uniformly chosen random sets from $\supp(g)$. 
\begin{theorem}\label{thm:DMK_two}
Assuming general assumptions (from Definition~\ref{def:general_as}), let $L\geq 2\cdot e$ and $s\geq 2$ additionally hold. If we have 
\begin{equation}\label{eq:codegree_Y}    
    \mathbb{P}\left(Y_2=y\right)\cdot r^{2-\frac{y}{k}}\cdot d^{\frac{y}{k}}\leq L^k
\end{equation} 
for all $y\in [k-1]$ then there exists $G\subseteq 2^X\setminus\{\emptyset\}$ with $\left<g\right>\subseteq \left<G\right>$ and $w\left(G,\frac{p}{L}\right)\leq 1$.
\end{theorem}

Both these theorems are generalizations of the result of DeMarco and Kahn from~\cite{TalagrandCliques} to the abstract setting of $k$-uniform hypergraphs, whereas~\cite{TalagrandCliques} studies only hypergraphs whose edges correspond to edge sets of cliques in graphs.

In the next section we provide more motivation for Defintion~\ref{def:general_as} and some preliminary results. We prove Theorem~\ref{thm:DMK_one} and Theorem~\ref{thm:DMK_two} in Sections~\ref{sec:DMK_one} and~\ref{sec:DMK_two} respectively. 
Finally, we demonstrate the versatility of Theorems~\ref{thm:DMK_one} and~\ref{thm:DMK_two} by applying them to the situation when the edges of a hypergraph correspond to the edge sets of hypergraph cliques in Section~\ref{sec:cliques_hyp}.

\section{Preliminary results}
In this section we aim to provide some motivation for Definition~\ref{def:general_as} and to also prove some first observations. 
We start with the setup of Conjecture~\ref{conj:Talagrand_rest}. Let $g\colon \binom{X}{k}\to[0,1]$ be given and let $p$ be such that $w(g,p)=1$ holds. We set $E:=\supp(g)$ and $d:=|E|$ and we assume that there exists an $r>0$ such that $g(S)=\frac{1}{r}\cdot \mathds{1}_{S\in E}$, that is, the function $g$ is constant on its support. Furthermore we may assume that $L^k\le r\le d$ holds, as otherwise $\left< g\right>$ is either empty (if $r>d$) or we just set $G:=E$ (if $r<L^k$). We will also assume throughout the paper that $k\le \ln(|X|)$, since the other case is taken care by our result in~\cite[Theorem~10]{FP23}. 

One further advantage to work in uniform regime which we are going to use is that  $w(g,p)=1$ allows to rewrite $p$ as 
$p = \left(\sum_{S\in \binom{X}{k}} g\left(S\right)\right)^{-\frac{1}{k}}$. 
 In the case $g$ is constant on its support further simplifies $p$ to be $p=\left(\frac{r}{d}\right)^{1/k}$. We further set $n:=|X|$ and let $m$ be the number of minimal elements in $\left< g\right>$. 

In the following we will define $G\subseteq 2^X$ as a collection of unions of randomly chosen sets from $E$ and we will provide conditions when such a (random) set $G\subseteq 2^X$  has small weight and satisfies $\left< g\right>\subseteq\left< G\right>$. We pick two values $s$, $t\in \NN$ and, for $i\in[t]$, $j\in [s]$, pick independently and uniformly at random sets $e_{i,j}\in E$ (with replacement). Next we define the (random) sets $e_i:=\bigcup_{j\in[s]} e_{i,j}$ for each $i\in [t]$ and the random family $G:=\left\{e_i \mid i\in[t]\right\}$.

We will see  later that $s\cdot k =\ln(n)$  and $t=p^{-s\cdot k}\cdot n$ will be good choices, and we ignore floor and ceiling signs as these will not affect our calculations.

With the assumptions above we obtain.

\begin{lemma}\label{lem:find_t}
If $t\geq p^{-s\cdot k}\cdot \ln(2m)$, then with probability at least $\frac{1}{2}$ we have $\left<g\right>\subseteq \left<G\right>$. 
In particular this holds for $t\geq p^{-s\cdot k}\cdot n$.
\end{lemma}
\begin{proof}[Proof of Lemma~\ref{lem:find_t}]
    We can start by rephrasing the probability:
    \begin{alignat}{100}\label{eq:gG_prob}
        \mathbb{P}\left(\left<g\right>\subseteq \left<G\right>\right) = \mathbb{P}\left(\forall S\in\left<g\right>\exists i\in [t] : e_i\subseteq S\right) = \mathbb{P}\left(\forall S\in\left<g\right>\exists i\in [t] \ \forall j\in [s] : e_{i,j}\subseteq S\right)
    \end{alignat}
Let $S\in\left<g\right>$. Since $g$ equals $\frac{1}{r}$ on its support, we have $1\leq \sum_{T\in\binom{S}{k}} g(T) = \frac{\left|\supp(g)\cap \binom{S}{k}\right|}{r}$. Hence, for any pair $(i,j)\in[t]\times[s]$, we obtain the following lower bound on  $\mathbb{P}\left(e_{i,j}\subseteq S\right)\geq \frac{r}{d} = p^k$ (where we also used $w(g,p)=\frac{d}{r}p^k=1$). Therefore this yields $\mathbb{P}\left(\forall j\in [s] : e_{i,j}\subseteq S\right)\geq p^{s\cdot k}$, and hence $\mathbb{P}\left(\forall i\in[t]\ \exists j\in [s] : e_{i,j}\not\subseteq S\right)\leq \left(1-p^{s\cdot k}\right)^t$. 

A union bound over all minimal $S\in\left<g\right>$ implies (with~\eqref{eq:gG_prob}):
    \begin{alignat*}{100}
        \mathbb{P}\left(\left<g\right>\subseteq \left<G\right>\right) & = \mathbb{P}\left(\forall S\in\left<g\right>\exists i\in [t] \ \forall j\in [s] : e_{i,j}\subseteq S\right)\\
        & \geq 1-m\cdot \left(1-p^{s\cdot k}\right)^t\geq 1-m\cdot \left(e^{-p^{s\cdot k}}\right)^t\\
        & \geq 1-m\cdot e^{-p^{s\cdot k}\cdot t} = 1 - e^{\ln(m)-p^{s\cdot k}\cdot t}.
    \end{alignat*}
For $t\geq p^{-s\cdot k}\cdot \ln(2m)$ this gives $\mathbb{P}\left(\left<g\right>\subseteq \left<G\right>\right)\ge 1-e^{-\ln(2)}=1/2$. Since the number of minimal elements of any set of $2^X$ is an antichain, it follows by Sperner's theorem that $m\le \binom{n}{\lfloor n/2\rfloor }\le 2^{n-1}$ and hence the case when $t\geq p^{-s\cdot k}\cdot n$.
\end{proof}

In the lemma above, $G$ is a random subset of $2^X$, formed by taking $t$ random unions of  $s$ random sets from $\supp(g)$. If, additionally, 
$\mathbb{E}\left[w\left(G,\frac{p}{L}\right)\right]<\frac{1}{2}$, then we can use
\begin{equation}\label{eq:cond_Exp}
\EE\left[w\left(G,\frac{p}{L}\right) \,\big|\, \left<g\right>\subseteq \left<G\right>\right]\le \frac{\EE\left[w\left(G,\frac{p}{L}\right)\right]}{\mathbb{P}(\left<g\right>\subseteq \left<G\right>)}
\end{equation}
to infer $\EE\left[w\left(G,\frac{p}{L}\right) \,\big|\, \left<g\right>\subseteq \left<G\right>\right]\le 1$. This implies the 
existence of  $G\subseteq 2^X$ with $\left<g\right>\subseteq \left<G\right>$ and $w\left(G,\frac{p}{L}\right)\le 1$. 
 Thus, we should be interested in computing $\mathbb{E}\left[w\left(G,\frac{p}{L}\right)\right]$:
\begin{alignat}{100}\label{eq:e_to_Y}   
    \mathbb{E}\left[w\left(G,\frac{p}{L}\right)\right] = \mathbb{E}\left[\sum_{i=1}^t \left(\frac{p}{L}\right)^{\left|e_i\right|}\right] = t\cdot\mathbb{E}\left[\left(\frac{p}{L}\right)^{\left|e_1\right|}\right],
\end{alignat}
which boils down to knowing the distribution of $|e_1|$. In the case $s=1$ it is easy (since each $e_i$ is a randomly chosen set from $\supp(g)$) and yields the following.

\begin{lemma}\label{lem:s_one} 
Let $c>1$. Then, for $L\geq 2c$ and $k\geq \log_c(r) + \log_c\left(\ln(n)\right)$, there exists $G$ with $\left<g\right>\subseteq \left<G\right>$ and $w\left(G,\frac{p}{L}\right)\leq 1$.
\end{lemma}
\begin{proof}
The number of minimal sets of $\supp(g)$ is at most $\binom{d}{r}$, since each set $S\in\left<g\right>$ has to contain at least $r$ many $k$-sets from $\supp(g)$ (indeed, we have $\sum_{T\in\binom{S}{k}} g(S)=\sum_{T\in\binom{S}{k}\cap \supp(g)}\frac{1}{r} \ge1$).

We set $s=1$. Then Lemma~\ref{lem:find_t}, applied with $t= p^{-s\cdot k}\cdot \ln(2m)$, yields $\PP\left[\left<g\right>\subseteq \left<G\right>\right]\ge 1/2$.  
Since  $2m\leq 2\binom{d}{r}\leq  d^r\leq \binom{n}{k}^r\leq n^{k\cdot r}$, we compute
    \begin{alignat*}{100}
        \mathbb{E}\left[w\left(G,\frac{p}{L}\right)\right]=t\cdot\mathbb{E}\left[\left(\frac{p}{L}\right)^{\left|e_1\right|}\right] = p^{-k}\cdot \ln(2m) \cdot \left(\frac{p}{L}\right)^{k} \le \frac{\ln(n^{k\cdot r})}{L^k}\leq \frac{k\cdot r\cdot \ln(n)}{(2c)^k}\leq \frac{k}{2^k}\cdot \frac{r\cdot \ln(n)}{c^k}\leq \frac{1}{2}.
    \end{alignat*}
The claim follows with~\eqref{eq:cond_Exp}, which was explained above.    
\end{proof}

From now on we always assume that $c=e$, $L\ge 2e$ and $k<\ln(r)+\ln\left(\ln(n)\right)$, and for simplicity, we always take  $t=p^{-s\cdot k}\cdot n$ (for some given $s\in\NN$). Given a (random) $G$, we define for each $j\in[s]$ the random variable $Y_j:=\left|e_{1,j}\cap \bigcup_{j'<j} e_{1,j'}\right|$ and observe that $|e_1|=s\cdot k - \sum_{j\leq s} Y_j$ holds. The random variables $Y_j$ control how the union of random $k$-element sets from $\supp(g)$ evolves ($Y_j$ are precisely the size overlaps with the union of the previously chosen $j-1$ $k$-element sets). Since each $e_i\in G$ has the same distribution, we concentrate on $e_1$. We can compute for $t=p^{-s\cdot k}\cdot n$ as in~\eqref{eq:e_to_Y}:
\begin{alignat*}{100}
        \mathbb{E}\left[w\left(G,\frac{p}{L}\right)\right] =t\cdot\mathbb{E}\left[\left(\frac{p}{L}\right)^{\left|e_1\right|}\right] = p^{-s\cdot k}\cdot n\cdot \mathbb{E}\left[\left(\frac{p}{L}\right)^{s\cdot k-\sum_{j\leq s} Y_j}\right] = \frac{n}{L^{s\cdot k}}\cdot \mathbb{E}\left[\left(\frac{p}{L}\right)^{-\sum_{j\leq s} Y_j}\right]
\end{alignat*}
The choice of $s\cdot k=\ln(n)$ takes care of $n$ above and we are left to analyze the distribution of $\sum_{j\leq s} Y_j$ or at least to provide sufficient conditions on $(Y_j)_{j\in[s]}$. The latter is done in the next two sections.

\section{Proof of Theorem~\ref{thm:DMK_one}}\label{sec:DMK_one}
\begin{proof}[Proof of Theorem~\ref{thm:DMK_one}]
We aim to show that, with $t=p^{-s\cdot k}\cdot n$ and the assumptions of the theorem, we get $\mathbb{E}\left[w\left(G,\frac{p}{L}\right)\mid \left<g\right>\subseteq \left<G\right>\right]\leq 1$, which would imply the existence of $G\subseteq 2^X\setminus\{\emptyset\}$ with the required properties $\left<g\right>\subseteq \left<G\right>$ and $w\left(G,\frac{p}{L}\right)\leq 1$. 

Lemma~\ref{lem:find_t} asserts $\mathbb{P}(\left<g\right>\subseteq \left<G\right>)\ge 1/2$. We use~\eqref{eq:cond_Exp} and~\eqref{eq:e_to_Y} to estimate $\mathbb{E}\left[w\left(G,\frac{p}{L}\right)\mid \left<g\right>\subseteq \left<G\right>\right]$ as follows:
    \begin{alignat*}{100}
        \mathbb{E}\left[w\left(G,\frac{p}{L}\right)\mid \left<g\right>\subseteq \left<G\right>\right] & \leq \frac{\mathbb{E}\left[w\left(G,\frac{p}{L}\right)\right]}{\mathbb{P}\left(\left<g\right>\subseteq \left<G\right>\right)}\leq 2\cdot \mathbb{E}\left[w\left(G,\frac{p}{L}\right)\right] = 2\cdot\mathbb{E}\left[\sum_{i=1}^t \left(\frac{p}{L}\right)^{\left|e_i\right|}\right] \\
        & = 2t\cdot\mathbb{E}\left[\left(\frac{p}{L}\right)^{\left|e_1\right|}\right] = 2\cdot p^{-s\cdot k}\cdot n\cdot \mathbb{E}\left[\left(\frac{p}{L}\right)^{s\cdot k-\sum_{j\leq s} Y_j}\right] \\
        & = 2\cdot\frac{n}{L^{s\cdot k}}\cdot \mathbb{E}\left[\left(\frac{p}{L}\right)^{-\sum_{j\leq s} Y_j}\right]\\
        & \overset{s\cdot k\geq \ln(n)}{\leq} 2\cdot\left(\frac{e}{L}\right)^{s\cdot k}\cdot \sum_{m=0}^{(s-1)\cdot k}\mathbb{P}\left(\sum_{j\leq s} Y_j= m\right)\cdot\left(\frac{p}{L}\right)^{-m}.
    \end{alignat*}        
Next we use $p=\left(\frac{r}{d}\right)^{1/k}$, $k\ge 2$, $L\ge 2\cdot e$ and the assumption~\eqref{eq:thm_one} to further simplify the last quantity as follows:
\begin{alignat*}{100}
    2\cdot\left(\frac{e}{L}\right)^{s\cdot k}\cdot \sum_{m=0}^{(s-1)\cdot k}\mathbb{P}\left(\sum_{j\leq s} Y_j =  m\right)\cdot\left(\frac{p}{L}\right)^{-m} = & \frac{2\cdot e^k}{L^k}\cdot\sum_{m=0}^{(s-1)\cdot k} \left(\frac{e}{L}\right)^{(s-1)\cdot k-m}\cdot \mathbb{P}\left(\sum_{j\leq s} Y_j= m\right)\cdot \left(\frac{r}{d\cdot e^k}\right)^{-\frac{m}{k}}\\
    \overset{L\geq 2e}{\leq} & \frac{1}{2} \sum_{m=0}^{(s-1)\cdot k} \left(\frac{e}{L}\right)^{(s-1)\cdot k -m}\cdot \mathbb{P}\left(\sum_{j\leq s} Y_j= m\right)\cdot \left(\frac{d\cdot e^k}{r}\right)^{\frac{m}{k}}\\
    \overset{\eqref{eq:thm_one}}{\le} & \frac{1}{2}\sum_{m=0}^{(s-1)\cdot k} \left(\frac{e}{L}\right)^{(s-1)\cdot k-m}\cdot \left(\frac{L}{2e}\right)^{(s-1)\cdot k -m}\\
    = & \frac{1}{2}\sum_{m=0}^{(s-1)\cdot k} 2^{m-(s-1)\cdot k} \leq 1.
\end{alignat*}
This proves the first part of Theorem~\ref{thm:DMK_one}.

For the second part we need to verify that~\eqref{eq:Y_cond} implies~\eqref{eq:thm_one}.  We first observe that 
\[
\mathbb{P}\left(Y_j = y_j\mid Y_1,\ldots,Y_{j-1}\right)\leq\mathbb{P}\left(Y_s \geq y_j\mid Y_1,\ldots,Y_{s-1}\right)
\]
 holds (by the definition of $Y_j$'s and since each $e_{1,j}$ has the same distribution as $e_{1,s}$).
Also, by the definition, we have $Y_1=0$. Now we estimate
    \begin{alignat*}{100}
        \mathbb{P}\left(\sum_{j\leq s} Y_j= m\right)\cdot \left(\frac{d\cdot e^k}{r}\right)^{\frac{m}{k}} & = \sum_{(y_j)_{2\leq j\leq s}\in \NN_0^{s-1}\colon\sum_j y_j=m} \mathbb{P}\left(Y_2=y_2,\ldots,Y_s=y_s\right)\cdot \left(\frac{d\cdot e^k}{r}\right)^{\frac{m}{k}}\\
        & = \sum_{(y_j)_{2\leq j\leq s}\colon\sum_j y_j=m}\left( \prod_{j=2}^s \mathbb{P}\left(Y_j=y_j \mid Y_2=y_2,\ldots Y_{j-1}=y_{j-1}\right) \right)\cdot \left(\frac{d\cdot e^k}{r}\right)^{\frac{m}{k}}\\
        & = \sum_{(y_j)_{2\leq j\leq s}\colon\sum_j y_j=m} \prod_{j=2}^s 
        \left(\mathbb{P}\left(Y_j=y_j \mid Y_2=y_2,\ldots Y_{j-1}=y_{j-1}\right)\cdot \left(\frac{d\cdot e^k}{r}\right)^{\frac{y_j}{k}}\right)\\
        & \leq \sum_{(y_j)_{2\leq j\leq s}\colon\sum_j y_j=m} \prod_{j=2}^s \mathbb{P}\left(Y_s\geq y_j \mid Y_2=y_2 \ldots, Y_{s-1}=y_{s-1}\right)\cdot \left(\frac{d\cdot 2^k\cdot e^k}{r}\right)^{\frac{y_j}{k}}\cdot \frac{1}{2^{y_j}}\\
        & \overset{\eqref{eq:Y_cond}}{\leq} \sum_{(y_j)_{2\leq j\leq s}\colon\sum_j y_j=m} \prod_{j=2}^s \left(\frac{L}{4e}\right)^{k-y_j}\cdot \frac{1}{2^{y_j}} = \sum_{(y_j)_{2\le j\le s}:\sum_j y_j=m} \left(\frac{L}{4e}\right)^{(s-1)\cdot k-m}\cdot \frac{1}{2^{m}}\\
        & = \sum_{(y_j)_{2\leq j\leq s}\colon\sum_j y_j=m} \left(\frac{L}{2e}\right)^{(s-1)\cdot k-m}\cdot \frac{1}{2^{(s-1)\cdot k}}\\
        & \leq 2^{(s-1)\cdot k}\cdot \left(\frac{L}{2e}\right)^{(s-1)\cdot k-m}\cdot \frac{1}{2^{(s-1)\cdot k}} = \left(\frac{L}{2e}\right)^{(s-1)\cdot k-m},
    \end{alignat*}
which implies that condition~\eqref{eq:thm_one} holds and thus finishes the second claim.
\end{proof}

\section{Proof of Theorem~\ref{thm:DMK_two}}\label{sec:DMK_two}
\begin{proof}[Proof of Theorem~\ref{thm:DMK_two}]
Recall that we set $E:=\supp(g)$.   This time we construct $G$ as $G_0\cup G_1$ by letting $G_0$ be the set of all unions of two different not disjoint $e$, $e'\in E$ and $G_1$ be the set of all unions of $r$ pairwise disjoint $\left(e_i\right)_{i\in [r]}\in E^r$ (we may assume that $r\in\NN$). Formally we define: 
    \begin{alignat*}{100}
        G_0 & := \left\{e\cup e' \mid e,e'\in E: |e\cap e'|\notin \{0,k\}\right\},\\
        G_1 & := \left\{\cup_{i\in [r]} e_i \,\big|\, \left(e_i\right)_{i\in [r]} \in E^{r}: \forall i\neq j: e_i\cap e_j =\emptyset\right\}.
    \end{alignat*}
    This covers all of $\left<g\right>$ since every $S\in \left<g\right>$ contains at least $r$ different $e\in E$ which are either disjoint (in which case $S\in G_1$) or have a pair of edges from $E$ that is not disjoint (in which case $S\in G_0$). Therefore we are left with calculating $w\left(G,\frac{p}{L}\right)$ as follows: 
    \begin{alignat*}{100}
        w\left(G,\frac{p}{L}\right) & \leq w\left(G_0,\frac{p}{L}\right) +w\left(G_1,\frac{p}{L}\right)
    \end{alignat*}
    Next we bound each term separately. First, we consider $w\left(G_0,\frac{p}{L}\right)$ (we use  $p=\left(\frac{r}{d}\right)^{1/k}$ and $d=|E|$):
    \begin{alignat*}{100}  
        w\left(G_0,\frac{p}{L}\right) & \leq \sum_{e,e'\in E: |e\cap e'|\notin \{0,k\}} \left(\frac{p}{L}\right)^{2k-\left|e\cap e'\right|}
       = \sum_{y=1}^{k-1} d^2\cdot\mathbb{P}\left(Y_2=y\right)\cdot \left(\frac{\left(\frac{r}{d}\right)^{\frac{1}{k}}}{L}\right)^{2k-y}\\
        & \leq \sum_{y=1}^{k-1} \frac{1}{L^{k-y}}\cdot \mathbb{P}\left(Y_2=y\right)\cdot \frac{r^{2-\frac{y}{k}}\cdot d^{\frac{y}{k}}}{L^k} \overset{\eqref{eq:codegree_Y}}{\leq} \sum_{y=1}^{k-1} \frac{1}{L^{k-y}} \leq \frac{1}{2}.
    \end{alignat*}
Second, we turn to $w\left(G_1,\frac{p}{L}\right)$ (and again we use  $p=\left(\frac{r}{d}\right)^{1/k}$):
    \begin{alignat*}{100}
        w\left(G_1,\frac{p}{L}\right) & \leq \sum_{\left(e_i\right)_{i\in [r]} \in E^{r}\colon \forall i\neq j\colon e_i\cap e_j =\emptyset} \left(\frac{p}{L}\right)^{r\cdot k} \leq \binom{d}{r}\left(\frac{p}{L}\right)^{r\cdot k} \leq \left(\frac{e\cdot d}{r}\right)^r\cdot \left(\frac{\left(\frac{r}{d}\right)^{\frac{1}{k}}}{L}\right)^{r\cdot k} = \frac{e^r}{L^{r\cdot k}} \leq \frac{1}{2}.
    \end{alignat*}
    Putting both estimates together yields:
    \begin{alignat*}{100}
        w\left(G,\frac{p}{L}\right) & \leq w\left(G_0,\frac{p}{L}\right) +w\left(G_1,\frac{p}{L}\right) \leq \frac{1}{2} + \frac{1}{2} \leq 1
    \end{alignat*}
\end{proof}

\section{Cliques in hypergraphs}\label{sec:cliques_hyp}
As an application of Theorems~\ref{thm:DMK_one} and~\ref{thm:DMK_two}, we generalize the result of DeMarco and Kahn~\cite{TalagrandCliques} to cliques in hypergraphs. Like in Definition \ref{def:general_as}, we collect our assumptions in the following definition.
\begin{definition}[\emph{Assumptions on clique hypergraphs}]\label{def:clique_as}
Let $\tn$, $\tk$, $\ell\in\NN$ with $\tn>\tk>\ell\ge 2$ and $\binom{\tk}{\ell}\le \ln\binom{\tn}{\ell}$. Set $X:=\binom{[\tn]}{\ell}$ and $n:=|X|=\binom{\tn}{\ell}$. Set $E:=\left\{\binom{T}{\ell}\,\bigg|\, T\in\binom{[\tn]}{\tk}\right\}$ and thus $|E|=\binom{\tn}{\tk}$ and we may view the tuple $(X,E)$ as a $k$-uniform hypergraph with $k=\binom{\tk}{\ell}$. Let $L\ge 1$ and $r\ge 1 $ with  $L^k\le r\le \binom{\tn}{\tk}$. Let $g\colon \binom{X}{k}\to[0,1]$ be defined through $g(S):=\frac{1}{r}\cdot\mathds{1}_{E}(S)$  and let $p\in[0,1]$ be such that $w(g,p)=1$. 

Then $G$ is constructed randomly by choosing $s$, $t\in \NN$ and, for $i\in[t]$, $j\in [s]$, picking independently and uniformly at random sets $e_{i,j}\in E$ (with replacement). With this we define the sets $e_i:=\bigcup_{j\in[s]} e_{i,j}$ for each $i\in [t]$ and get the (random) set $G:=\left\{e_i \mid i\in[t]\right\}$. Additionally we define for each $j\in[s]$ the random variable $Y_j:=\left|e_{1,j}\cap \bigcup_{j'<j} e_{1,j'}\right|$. The sets $e_{1,1}$, \ldots, $e_{1,s}\in E$, correspond to sets $T_1$, \ldots, $T_s\in \binom{[\tn]}{\tk}$ chosen uniformly at random (where each  $e_{1,i}$ is a clique spanned on $T_i$ in $K^{(\ell)}_{\tn}$). For each $j\in[s]$, we define  the random variables $\tY_1$,\ldots, $\tY_s$, where $\tY_j:=\left|T_{j}\cap \bigcup_{j'<j} T_{j'}\right|$. 
\end{definition}

For a nonnegative integer $y$, let $\ty$ be the smallest integer with $\binom{\ty}{\ell}\geq y$. 
By definitions of $(Y_j)_{j\in[s]}$ and $(\tY_j)_{j\in[s]}$ we immediately obtain
\begin{equation}\label{eq:Y_tY}
\mathbb{P}\left(Y_s \geq y\mid Y_1,\ldots,Y_{s-1}\right)\leq \mathbb{P}\left(\tY_s\geq \ty\mid Y_1,\ldots,Y_{s-1}\right).
\end{equation}
 
\begin{lemma}\label{lem:DMK_one}
Given assumptions on clique hypergraphs (from Definition~\ref{def:clique_as}), the following holds. 
\begin{enumerate}
    \item For $\ell\le 4\ln(\tn)$, $r\ge 2^{3k}\cdot e^{3k}\cdot\ln^{2\tk}(\tn)$ and $L\geq 4\cdot e$, Theorem~\ref{thm:DMK_one} is applicable with $s\cdot k=\ln(n)$.
    \item For $\tk=\ell+1$, $r\ge \left(\ln(e\tn/\ell)\right)^\ell\cdot \tn \cdot  e^{3(\ell+1)}$ and $L\geq 4\cdot e$, Theorem~\ref{thm:DMK_one} is applicable with  $s\cdot k=\ln(n)$.
\end{enumerate}
\end{lemma}
\begin{proof}[Proof of Lemma~\ref{lem:DMK_one}]
Since we have $r\le \binom{\tn}{\tk}$ and $\binom{\ty}{\ell}\geq y$, we estimate
    \begin{alignat}{100}\label{eq:prob_DMK_one}
        \mathbb{P}\left(Y_s \geq y\mid Y_1,\ldots,Y_{s-1}\right)\left(\frac{\binom{\tn}{\tk}\cdot 2^k\cdot e^k}{r}\right)^{\frac{y}{k}} \overset{\eqref{eq:Y_tY}}{\le} \mathbb{P}\left(\tY_s\geq \ty\mid Y_1,\ldots,Y_{s-1}\right)\left(\frac{\binom{\tn}{\tk}\cdot 2^{k}\cdot e^{k}}{r}\right)^{\frac{\binom{\ty}{\ell}}{\binom{\tk}{\ell}}}.
    \end{alignat}
    Since the cliques (associated with $Y_i$s) are chosen uniformly (and independently) we can just say that each clique draws $\tk$ vertices (without replacement) independently of the others. If we now assume the worst case which is that all the previous cliques are disjoint, they still at most span $s\cdot \tk\le s\cdot k= \ln(n)=\ln\binom{\tn}{\ell}\leq \ell\cdot \ln(\tn) \leq 4\ln^2(\tn)$ many vertices. 
Let $T_1$, \ldots, $T_{s}$ be randomly chosen $\tk$-element sets from $[\tn]$. We set $\tm:=\left|\cup_{i=1}^{s-1} T_i\right|$. The random variable $\tY_s=\left|T_s\cap\left(\cup_{i=1}^{s-1} T_i\right)\right|$ has hypergeometric distribution. Therefore we get (we write $(x)_u$ for the falling factorial $\prod_{i=0}^{u-1} (x-i)$):
    \begin{alignat}{100}\label{eq:tYs_ge_ty}
        \mathbb{P}\left(\tY_s\geq \ty\mid Y_1,\ldots,Y_{s-1}\right) & \leq \sum_{j= \ty}^{\tk} \frac{\binom{\tm}{j}\binom{\tn-\tm}{\tk-j}}{\binom{\tn}{\tk}} = 
        \sum_{j= \ty}^{\tk} \frac{(\tm)_{\ty} (\tk)_{\ty}\binom{\tm-\ty}{j-\ty}\binom{\tn-\tm}{\tk-j}}{(\tn)_{\ty}(j)_{\ty}\binom{\tn-\ty}{\tk-\ty}}\\ \nonumber
        & \le \frac{(\tm)_{\ty}}{(\tn)_{\ty} }\binom{\tk}{\ty}\sum_{j=0}^{\tk-\ty} \frac{\binom{\tm-\ty}{j}\binom{\tn-\tm}{(\tk-\ty)-j}}{\binom{\tn-\ty}{\tk-\ty}}\\ \nonumber
        & \leq \left(\frac{\tm}{\tn}\right)^{\ty} \left(\frac{e\cdot \tk}{\ty}\right)^{\ty}\leq \left(\frac{4\cdot e\cdot\tk \cdot \ln^2(\tn)}{\tn}\right)^{\ty},
    \end{alignat}
 where we used $\tm\le 4 \ln^2(\tn)$ in the last inequality.   
 We thus further bound the right hand side of~\eqref{eq:prob_DMK_one} by
    \begin{alignat*}{100}
\left(\frac{4\cdot e\cdot\tk \cdot \ln^2(\tn)}{\tn}\right)^{\ty} \left(\frac{\binom{\tn}{\tk}\cdot 2^{k}\cdot e^{k}}{r}\right)^{\frac{\binom{\ty}{\ell}}{\binom{\tk}{\ell}}}
& \le \left(\frac{4\cdot e\cdot\tk \cdot \ln^2(\tn)}{\tn}\right)^{\ty}\left(\frac{\tn^{\tk}\cdot e^{\tk}\cdot 2^{k}\cdot e^{k}}{\tk^{\tk}\cdot r}\right)^{\frac{\ty}{\tk}} \\
&=\left(\frac{2^{2\tk}\cdot e^{\tk}\cdot \tk^{\tk}\cdot \ln^{2\tk}(\tn)\cdot \tn^{\tk}\cdot e^{\tk}\cdot 2^{k}\cdot e^{k}}{\tn^{\tk}\cdot \tk^{\tk}\cdot r}\right)^{\frac{\ty}{\tk}}
        \le \left(\frac{e^{2\tk}\cdot \ln^{2\tk}(\tn)\cdot 2^{3k}\cdot e^{k}}{r}\right)^{\frac{\ty}{\tk}}\\
        & \leq 1 \leq \left(\frac{L}{4e}\right)^{k-y},
    \end{alignat*}
where in the last two steps we used $L\geq 4e$, $r\geq 2^{3k}\cdot e^{3k}\cdot \ln^{2\tk}(\tn)$.   Thus, Theorem~\ref{thm:DMK_one} is applicable in this case.

Let now $\tk=\ell+1$ and therefore $k=\binom{\tk}{\ell}=\tk=\ell+1$. Then~\eqref{eq:tYs_ge_ty} can be simplified, since $\ty\in\{\ell,\ell+1\}$, as follows.
    \begin{alignat*}{100}
        \mathbb{P}\left(\tY_s\geq \ty\mid Y_1,\ldots,Y_{s-1}\right) & \leq \sum_{j= \ell}^{\ell+1} \frac{\binom{\tm}{j}\binom{\tn-\tm}{\tk-j}}{\binom{\tn}{\tk}} = 
        \frac{\binom{\tm}{\ell}(\tn-\tm)}{\binom{\tn}{\ell+1}}+        \frac{\binom{\tm}{\ell+1}}{\binom{\tn}{\ell+1}}\\ 
        & \le \frac{(\tm)_\ell (\tn-\tm)(\ell+1)}{(\tn)_{\ell+1}}+  \frac{(\tm)_{\ell+1}}{(\tn)_{\ell+1}} \le \frac{(\tm)_\ell (\ell+1)}{(\tn)_{\ell}}.
    \end{alignat*}
We now use the following estimate on $\tm$: $\tm\le s\cdot \tk\le s\cdot k= \ln(n)\le \ln\left(\left(\frac{e\tn}{\ell}\right)^\ell\right)= \ell\cdot \ln\left(e\tn/\ell\right)$. We  continue the estimate from~\eqref{eq:prob_DMK_one} as follows.
\begin{alignat*}{100}
        \mathbb{P}\left(Y_s \geq y\mid Y_1,\ldots,Y_{s-1}\right)\left(\frac{\binom{\tn}{\tk}\cdot 2^k\cdot e^k}{r}\right)^{\frac{y}{k}}  
        \overset{\eqref{eq:Y_tY}}{\le}  \frac{(\tm)_\ell\cdot (\ell+1)}{(\tn)_{\ell}} \cdot \left(\frac{\binom{\tn}{\ell+1}\cdot 2^{\ell+1}\cdot e^{\ell+1}}{r}\right)^{\frac{\binom{\ty}{\ell}}{\ell+1}}\\
          \overset{r\le \binom{\tn}{\tk}}{\le} \frac{(\tm)_\ell \cdot (\ell+1)}{(\tn)_{\ell}} \cdot \left(\frac{\binom{\tn}{\ell+1}\cdot 2^{\ell+1}\cdot e^{\ell+1}}{r}\right)
         \le \frac{(\tm)_\ell\cdot (\ell+1)}{(\tn)_{\ell}} \cdot \left(\frac{\tn^{\ell+1}\cdot e^{\ell+1} \cdot 2^{\ell+1}\cdot e^{\ell+1}}{ (\ell+1)^{\ell+1}\cdot r}\right)\\
         \le \frac{\tm^{\ell}}{\tn^\ell}\cdot \frac{\tn^{\ell+1} \cdot e^{2(\ell+1)} \cdot 2^{\ell+1}}{(\ell+1)^{\ell}\cdot r}\le 
         \frac{\ell^\ell \cdot\left(\ln(e\tn/\ell)\right)^\ell \cdot \tn \cdot  e^{2(\ell+1)} \cdot 2^{\ell+1}}{(\ell+1)^{\ell}\cdot r}\le \frac{\left(\ln(e\tn/\ell)\right)^\ell\cdot \tn \cdot  e^{3(\ell+1)} }{r } \le 1. 
\end{alignat*}
Again, since $L\geq 4e$, we have $\left(\frac{L}{4e}\right)^{k-y}\ge 1$ and Theorem~\ref{thm:DMK_one} is also applicable in this case.

\end{proof}

\begin{lemma}\label{lem:DMK_two}
Given assumptions on clique hypergraphs (from Definition~\ref{def:clique_as}),  the following holds.  For 
\[
r\leq \sqrt{\frac{L^k\cdot \tn^{\ell-1}}{e^{2\tk}\cdot \tk^{\ell-1}}}
\]
 and $L\geq 2\cdot e$, Theorem~\ref{thm:DMK_two} is applicable.
\end{lemma}
\begin{proof}[Proof of Lemma~\ref{lem:DMK_two}]
    This time we only have to calculate the probability for the overlap of two randomly chosen hypergraph cliques each on $\tk$ vertices. Since two distinct, not disjoint cliques intersect in a smaller clique, we write $y=\binom{\ty}{\ell}$ for the number of edges in their intersection (for some integer $\ty$ with $\tk-1\ge \ty\ge \ell$). It follows
    \begin{alignat*}{100}
        \mathbb{P}\left(Y_2=y\right)\cdot r^{2-\frac{y}{k}}\cdot d^{\frac{y}{k}} & = \frac{\binom{\tk}{\ty}\binom{\tn-\tk}{\tk-\ty}}{\binom{\tn}{\tk}}\cdot r^{2-\frac{y}{k}}\cdot \binom{\tn}{\tk}^{\frac{\binom{\ty}{\ell}}{\binom{\tk}{\ell}}} \le \binom{\tk}{\ty}\cdot \left(\frac{\tk}{\tn}\right)^{\ty}\cdot r^2\cdot \left(\frac{e\cdot \tn}{\tk}\right)^{\tk\cdot \frac{\binom{\ty}{\ell}}{\binom{\tk}{\ell}}}\\
        & \leq e^{2\tk}\cdot r^2\cdot \left(\frac{\tk}{\tn}\right)^{\ty-\tk\cdot \frac{\binom{\ty}{\ell}}{\binom{\tk}{\ell}}},
    \end{alignat*}
where the last bound is maximized for $\ty=\ell$ or $\ty=\tk-1$, since $f(\ty):=\ty-\tk\cdot \frac{\binom{\ty}{\ell}}{\binom{\tk}{\ell}}$ is concave for $\ty\ge \ell$. %no need to prove concavity!
 We have $f(\ell)=\ell-\tk\cdot\frac{\binom{\ell}{\ell}}{\binom{\tk}{\ell}} = \ell-\tk\cdot\frac{1}{\binom{\tk}{\ell}} \ge \ell - 1$    and 
 $f(\tk-1)=(\tk-1)-\tk\cdot\frac{\binom{\tk-1}{\ell}}{\binom{\tk}{\ell}} = (\tk-1)-(\tk-\ell) = \ell - 1$.
 We thus obtain with $r\leq \sqrt{\frac{L^k\cdot \tn^{\ell-1}}{e^{2\tk}\cdot \tk^{\ell-1}}}$
    \begin{alignat*}{100}
        \mathbb{P}\left(Y_2=y\right)\cdot r^{2-\frac{y}{k}}\cdot d^{\frac{y}{k}} & \leq e^{2\tk}\cdot r^2\cdot \left(\frac{\tk}{\tn}\right)^{\ty-\tk\cdot \frac{\binom{\ty}{\ell}}{\binom{\tk}{\ell}}} \leq e^{2\tk}\cdot r^2\cdot \left(\frac{\tk}{\tn}\right)^{\ell-1}\leq L^k,
    \end{alignat*}
    which finishes the proof.
\end{proof}

\begin{corollary}\label{cor:hypergraphs}
Given assumptions on clique hypergraphs (from Definition~\ref{def:clique_as}),  the following holds for $L\geq 2^{12}\cdot e^{16}$ and $\tn$ large enough. There exists $G\subseteq 2^X\setminus\{\emptyset\}$ so that $\left<g\right>\subseteq \left<G\right>$ and $w(G,p/L)\le 1$.
\end{corollary}
\begin{proof}[Proof of Corollary~\ref{cor:hypergraphs}]
The claim follows by comparing the bounds on $r$ in Lemmas~\ref{lem:DMK_one} and~\ref{lem:DMK_two} and making sure that they cover all possible ranges of $r$.

We assume that $\ell\le 4\ln(\tn)$. By squaring both sides we see that $ 2^{3k}\cdot e^{3k}\cdot\ln^{2\tk}(\tn) \leq \sqrt{\frac{L^k\cdot \tn^{\ell-1}}{e^{2\tk}\cdot \tk^{\ell-1}}}$ is equivalent to 
\begin{equation}\label{eq:cor_bound}
\frac{2^{6k}\cdot e^{6k}\cdot e^{2\tk}}{L^k}\cdot \frac{\ln^{4\tk}(\tn)\cdot \tk^{{\ell-1}}}{\tn^{{\ell-1}}} \leq 1.
\end{equation}

If we have $\ln^{4\tk}(\tn)\le \tn^{(\ell-1)/2}$ then by taking the logarithm we get $\tk\le \frac{\ell-1}{8}\frac{\ln (\tn)}{\ln\ln(\tn)}\le \ln^2(\tn)\le \tn^{1/2}$ for $\tn$ large enough. Hence, inequality~\eqref{eq:cor_bound} is implied if 
$\frac{2^{6k}\cdot e^{6k}\cdot e^{2\tk}}{L^k}\le 1$, which yields the following sufficient condition for $L$: $L\ge 2^6\cdot e^8$.

If on the other hand $\ln^{4\tk}(\tn)\ge \tn^{(\ell-1)/2}$ then we have $\tk\ge \frac{\ell-1}{8}\frac{\ln (\tn)}{\ln\ln(\tn)}$. 
This implies $\ell\le \tk-2$ for $\tn$ large enough. 
With  $k=\binom{\tk}{\ell}\ge \frac{\tk^2}{4}$, 
$\tk\le k\le\ln(n)\le \ell\ln(\tn)$ and $L\ge e$ we bound 
\[
\frac{\ln^{4\tk}(\tn)}{L^{k/2}}\le \frac{\exp\left(4\ell\cdot\ln(\tn)\cdot\ln\ln(\tn)\right)}{\exp(\tk^2/8)}\le \frac{\exp\left(4\ell\cdot\ln(\tn)\cdot\ln\ln(\tn)\right)}{\exp\left(\frac{(\ell-1)^2}{2^9}\frac{\ln^2 (\tn)}{(\ln\ln(\tn))^2}\right)} \le 1
\]
for $\tn$ large enough. The inequality~\eqref{eq:cor_bound} now holds if $\frac{2^{6k}\cdot e^{6k}\cdot e^{2\tk}}{L^{k/2}}\le 1$, which yields the following sufficient condition for $L$: $L\ge 2^{12}\cdot e^{16}$.

Next we assume $\ell> 4\ln(\tn)$. The only case to consider here is $\tk=\ell+1$ and hence $k=\ell+1$. Indeed, if $\tk\ge \ell+2$ then this would yield  with  $k=\binom{\tk}{\ell}\ge \frac{\tk^2}{4}\ge \frac{\ell^2}{4}$ and $k\le \ln(n)\le \ell\ln(\tn)$ that $\ell\le 4\ln(\tn)$.  Thus, in the following we assume $\tk=k=\ell+1$ and use the second lower bound on $r$ from Lemma~\ref{lem:DMK_one}. Then
\[
\left(\ln(e\tn/\ell)\right)^\ell \cdot \tn \cdot  e^{3(\ell+1)} \le \sqrt{\frac{L^k\cdot \tn^{\ell-1}}{e^{2\tk}\cdot \tk^{\ell-1}}}
\]
is equivalent to 
\begin{equation}\label{eq:eq:cor_bound_sc}
\frac{\left(\ln(e\tn/\ell)\right)^{2\ell} \cdot  e^{8(\ell+1)}\cdot (\ell+1)^{\ell-1}}{L^{\ell+1}\cdot \tn^{\ell-3}} \le 1.
\end{equation}
Rewriting~\eqref{eq:eq:cor_bound_sc} as
\[
\frac{\left(\ln(e\tn/\ell)\right)^{2\ell}}{(e\tn/\ell)^\ell}\cdot\frac{(\ell+1)^{\ell-1}}{\ell^{\ell}}\cdot\frac{\tn^3}{e^{\ell+1}}\cdot\frac{1}{e\cdot (L/e^{10})^{\ell+1}}\le 1,
\]
we see that it suffices to take $L\ge e^{10}$ for the inequality~\eqref{eq:eq:cor_bound_sc} to hold.

Thus, we have shown Corollary  for $L\geq 2^{12}\cdot e^{16}$ and $\tn$ large enough.
\end{proof}

\end{document}